\documentclass[12pt,a4paper]{amsart}

\usepackage{amsmath,amsfonts,amssymb,mathrsfs}
\usepackage[francais]{babel}
\usepackage[utf8]{inputenc}
\usepackage[T1]{fontenc}
\usepackage{graphicx}
\usepackage{tabularx}
\newcommand{\nocontentsline}[3]{}
\newcommand{\tocless}[2]{\bgroup\let\addcontentsline=\nocontentsline#1{#2}\egroup}
\DeclareUnicodeCharacter{0177}{\^u}

\newtheorem{theo}{Th\'eor\`eme}[section]

\newtheorem{lemme}[theo]{Lemme}

\newtheorem{proposition}[theo]{Proposition}
\newtheorem{corollaire}[theo]{Corollaire}

\newtheorem{remarque}[theo]{Remarque}

\begin{document}

\title[Images de représentations galoisiennes]{Images de représentations galoisiennes associées à certaines formes modulaires de Siegel de genre $2$}
\author{Salim Tayou}
\date\today
\maketitle

\tocless\subsection*{\textbf{Abstract}}. We study the image of the $\ell$-adic Galois representations associated to the four  vector valued  Siegel modular forms appearing in the work of  Chenevier and  Lannes \cite{chenevier}. These representations are symplectic of dimension $4$. Following  methods used by Dieulefait in \cite{dieulefait}, we determine the primes $\ell$ for which these representations are absolutely irreducible. In addition, we show that their image is "full" for all  primes $\ell$ such that the associated residual representation is absolutely irreducible, except in two cases.  
\tableofcontents

\section{Introduction}

Soient $j,k$ des entiers $\geq 0$ et ${\mathrm S}_{j,k}(\Gamma_2)$ l'espace des formes modulaires paraboliques de Siegel pour le groupe $\mathrm {Sp}_4(\mathbb{Z})=:\Gamma_2$ à coefficients vectoriels $\mathrm{Sym}^j \mathbb{C}^2 \otimes \mathrm{\det}^k$ (on dira aussi "de poids $(j,k)$"). Dans ce travail, on s'intéresse à la détermination des images des représentations galoisiennes $\ell$-adiques associées à quatre telles formes modulaires de poids dans la liste $\{(6,8), (8,8), (4,10), (12,6)\}$. Notre motivation provient du travail de Chenevier et Lannes sur le comptage des $p$-voisinages au sens de Kneser entre réseaux de Niemeier (voir \cite{chenevier}), dans lequel ces formes interviennent de manière essentielle.

On rappelle que l'espace $\mathrm{S}_{j,k}(\Gamma_2)$ est muni d'une action d'opérateurs de Hecke (voir par exemple le chapitre $16$ de \cite{123})). Soit $f \in {\rm S}_{j,k}(\Gamma_2)$ une forme propre pour l'action de ces opérateurs ($f$ non nulle). Le sous-corps de $\mathbb{C}$ engendré par les valeurs propres de l'action sur $f$ des opérateurs en question est alors un corps de nombres, que l'on notera $\mathbb{Q}(f)$. De plus, on associe suivant Andrianov un produit Eulérien $\zeta_f(s) = \prod_p Q_p(p^{-s})^{-1}$, absolument convergent pour $\mathrm{Re}(s)$ assez grand, où pour tout premier $p$ l'élément $Q_p(X) \in \mathbb{Q}(f)[X]$ est un polynôme de degré $4$ de la forme 
$$Q_p(X)  = 1 - a_p(f)  X + b_p(f) X^2 - p^{2k+j-3} a_p(f) X^3 + p^{4k+2j-6} X^4,$$
les coefficients $a_p(f)$ et $b_p(f)$ étant les valeurs propres de certains opérateurs de Hecke explicites.\\

Le théorème suivant est essentiellement dû à Weissauer \cite{weissaur} (voir également la discussion du chapitre VIII.2.15 de \cite{chenevier}).

\begin{theo}\label{chenevier1} 
Soit $f \in \mathrm {S}_{j,k}(\Gamma_2)$ une forme propre pour l'action des opérateurs de Hecke. Pour tout nombre premier $\ell$ et toute place $\lambda$ de $\mathbb{Q}(f)$ au dessus de $\ell$, il existe un $\overline{\mathbb{Q}(f)_\lambda}$-espace vectoriel $V$ de dimension $4$, muni d'une forme bilinéaire alternée non dégénérée $b$, ainsi qu'une représentation semi-simple et continue $$r_f : \mathrm{Gal}(\overline{\mathbb{Q}}/\mathbb{Q}) \rightarrow \mathrm{GSp}(b)$$
non ramifiée hors de $\ell$ et vérifiant les propriétés suivantes : \\
\begin{itemize}
\item[(i)] pour tout nombre premier $p \neq \ell$, $\det(1 - X r_f(\mathrm {Frob}_p)) = Q_p(X)$,
\item[(ii)] si $\eta : \mathrm{GSp}(b) \rightarrow \overline{\mathbb{Q}(f)_\lambda}^\times$ est le facteur de similitude associé à $b$, alors on a $\eta \circ r_f = \chi_\ell^{j+2k-3}$, $\chi_\ell$ désignant le caractère cyclotomique $\ell$-adique. 
\end{itemize}
\end{theo}
%
%

On rappelle que Tsushima a déterminé dans \cite{tsuchima} une formule explicite pour la dimension de l'espace $\mathrm{S}_{j,k}(\Gamma_2)$ pour $k\geq 5$. Cette dimension vaut $1$ pour les quatre couples $(j,k)$ dans $\{(6,8), (8,8), (4,10), (12,6)\}$ déjà mentionnés. Si $f$ est un élément non nul de $\mathrm{S}_{j,k}(\Gamma_2)$ pour l'une de ces quatre valeurs, il est en particulier automatiquement vecteur propre de tous les opérateurs de Hecke. De plus, on a alors  $\mathbb{Q}(f)=\mathbb{Q}$ et les polynômes $\mathrm{Q}_p$ associés à $f$ sont dans $\mathbb{Z}[X]$ d'après la proposition IX.1.9 de \cite{chenevier}. On notera $r_{j,k,\ell}$ la représentation $\ell$-adique $r_f$ associée à cette forme $f$ et à $\ell$ par le théorème ci-dessus. C'est une représentation irréductible d'après un résultat de Calegari et Gee \cite{calegari} et le \S X.1.3 de \cite{chenevier}.

On notera également $\overline{r_{j,k,\ell}}$ la $\overline{\mathbb{F}_\ell}$-représentation résiduelle associée à $r_{j,k,\ell}$ par un procédé standard (voir par exemple la discussion à la fin du chapitre X.1 de \cite{chenevier}). Disons brièvement que la représentation $r_{j,k,\ell}$ est définie sur une extension finie $E$ de $\mathbb{Q}_\ell$ et admet un $\mathcal{O}_E$-réseau stable $L$. On définit $\overline{r_{j,k,\ell}}$  comme étant la semi-simplifiée de la représentation $L/(\pi_E L) \otimes_{\mathcal{O}_E/\pi_E} \overline{\mathbb{F}_\ell}$, où $\pi_E$ désigne une uniformisante de $\mathcal{O}_E$. Elle est en fait définie sur le corps fini $\mathbb{F}_\ell$ et peut être choisie à valeurs dans le groupe fini $\mathrm{ GSp}_4(\mathbb{F}_\ell)$ (voir {\it loc. cit.}). Elle est continue, non-ramifiée hors de $\ell$, et vérifie des conditions similaires aux (i) et (ii) de l'énoncé  ci-dessus, dans lesquelles $Q_p$ et $\chi_\ell$ sont remplacés par leurs réductions modulo $\ell$.

\begin{theo}\label{irreductible}
Soit $\ell$ un nombre premier. Alors la représentation  $\overline{r_{j,k,\ell}}$ est absolument irréductible si, et seulement si, on a $\ell\geq 7$ et si l'on est dans l'un des cas suivants:
$$(j,k)=(6,8) \quad \text{et} \quad \ell\neq11,17, \qquad (j,k)=(8,8) \quad \text{et}\quad  \ell\neq 13, 17,23, $$
$$(j,k)=(12,6) \quad\text{et} \quad \ell\neq 7,13,19, \quad (j,k)=(4,10) \quad\text{et}\quad \ell\neq 11,19,41.$$

\end{theo}
Une partie des résultats de ce théorème est déjà connue d'après le théorème X.4.4 et la prop. X.4.10 de \cite{chenevier} pour des petites valeurs de $\ell$ ($\ell<j+2k-3$) et pour $(8,8,23)$ et $(4,10,41)$. Pour $\ell$ assez grand, nous reprenons la méthode de Dieulefait dans \cite{dieulefait}, qui démontre que le théorème vaut pour tout $\ell$ assez grand, et ce avec un minorant effectif. Un examen de cette méthode, ainsi que la connaissance des polynômes caractéristiques de $r_{j,k,\ell}(\mathrm{Frob}_p)$ pour $p=2,3$ (déterminés par Faber et van der Geer dans \cite[\S 3]{123}, voir aussi \cite[tables C.3 et C.4]{chenevier} et \cite{megarbane}), nous permet de conclure pour tout triplet 
non couvert par les résultats de Chenevier et Lannes \cite{chenevier}.
	
Le théorème principal que l'on démontre dans cet article est le suivant:
\begin{theo}\label{images}
Soit $\ell$ un nombre premier. On suppose que $\overline{r_{j,k,\ell}}$ est absolument irréductible et que $(j,k,\ell)\neq(6,8,13),(4,10,17)$. Alors l'image de la représentation $r_{j,k,\ell}$ est conjuguée à $$\{M\in \mathrm{GSp}_{4}(\mathbb{Z}_{\ell})| \eta(M) \in ({\mathbb{Z}_{\ell}}^{\times})^{j+2k-3} \}.$$ 
\end{theo}
La démonstration de ce théorème procède de la manière suivante. Tout d'abord, on se ramène (par un argument dû à Serre) à montrer que l'image de $\overline{r_{j,k,\ell}}$ contient $\mathrm{Sp}_4(\mathbb{F}_\ell)$. Si ce n'était pas le cas, cette image serait incluse dans l'un des sous-groupes maximaux de $\mathrm {GSp}_4(\mathbb{F}_\ell)$ classifiés par Mitchell \cite{mitchell} (voir le \S 4). Nous élaborons pour cela la méthode de Dieulefait \cite{dieulefait} et dégageons des critères effectifs permettant d'exclure chacun des cas possibles. L'application de ces critères nécessite de connaître le polynôme caractéristique de $r_{j,k,\ell}(\mathrm{Frob}_p)$ pour $p \leq 13$, qui est connu d'après \cite[table C.3, C.4]{chenevier}. 
Mentionnons que le cas $(6,8,13)$ non traité par le théorème \ref{images} a déjà été mis en évidence dans les travaux de Chenevier et Lannes (voir \cite{chenevier} chapitre X. remarque 4.11 où une conjecture est faite). Le fait que le cas $(4,10,17)$ soit particulier semble nouveau. Les résultats des calculs laissent penser que l'image stabilise un plan non dégénéré dans $\Lambda^2V$. Il serait intéressant d'avoir une démonstration de ces résultats.

L'article s'organise comme suit: dans le premier paragraphe, on donne des rappels  sur les caractères fondamentaux ainsi que la description de la restriction de $\overline{r_{j,k,\ell}}$ à un sous-groupe d'inertie en $\ell$. 
Dans le deuxième paragraphe,  on démontre l'irréductibilité des représentations $\overline{r_{j,k,\ell}}$. Dans le troisième paragraphe, on énonce la liste des sous-groupes maximaux de $\mathrm{PGSp}(4,\mathbb{F}_{\ell})$. Enfin, le dernier paragraphe  démontre le théorème \ref{images}.\\

Les méthodes de cet article permettent sans doute d'obtenir des résultats analogues concernant les représentations galoisiennes $\ell$-adiques associées à des formes propres dans $\mathrm{S}_{j,k}(\Gamma_2)$ pour d'autres couples $(j,k)$ que ceux considérés ici, du moins lorsque la dimension de cet espace est $1$ et que $\ell$ est assez grand. Cependant, si l'on désire des résultats complets (incluant toutes les valeurs de $\ell$, notamment les petites) il faut pouvoir disposer de diverses congruences, comme celles "à la Harder" démontrées dans \cite[Thm 4.4]{chenevier}, ainsi que de tables de polynômes caractéristiques de Frobenius (il n'est pas certain que celles de Faber et van der Geer suffisent). Mégarbané nous a communiqué des résultats dans ce sens pour les formes vérifiant $j+2k-3 = 25$. 

\subsubsection*{Remerciements}Je tiens à remercier Gaëtan Chenevier pour sa relecture détaillée, ses remarques judicieuses et son encadrement pour l'écriture de cet article. Je remercie également Quentin Guignard et Thomas Mégarbané pour les corrections apportées à ce document.

\section{Action de l'inertie}
\subsubsection*{Caractères fondamentaux}
Une bonne référence pour cette partie est le premier paragraphe de l'article de Serre \cite{serre71}. On se donne  un nombre premier $\ell$, et on fixe une clôture algébrique $\overline{\mathbb{Q}_{\ell}}$ du corps des nombres $\ell$-adiques $\mathbb{Q}_{\ell}$ et $\overline{\mathbb{F}_{\ell}}$ du corps fini $\mathbb{F}_{\ell}$. 
On fixe un isomorphisme $\sigma : \mathbb{Z}_{\ell}^{nr}/\ell\mathbb{Z}_{\ell}^{nr} \rightarrow \overline{\mathbb{F}_{\ell}}$, où $\mathbb{Z}_{\ell}^{nr}$ désigne l'anneau des entiers de l'extension maximale non-ramifiée de $\mathbb{Q}_{\ell}$ dans $\overline{\mathbb{Q}_{\ell}}$ et que l'on note $\mathbb{Q}_{\ell}^{nr}$. Le groupe d'inertie en $\ell$ est  $\mathrm{I}_{\ell}=\mathrm{Gal}(\overline{\mathbb{Q}_{\ell}} /\mathbb{Q}_{\ell}^{nr})$.

Si $N>0$, on notera $x_N$ une racine $(\ell^N-1)$-ième de $\ell$  dans $\overline{\mathbb{Q}_{\ell}}$. On définit alors le caractère fondamental de niveau $N$ par la composée 
$$\psi_{N}: \mathrm{I}_{\ell} \rightarrow \mu_{\ell^N-1} \rightarrow  {\overline{\mathbb{F}_{\ell}}}^{\times} $$ donnée par $$\tau \mapsto \frac{\tau(x_N)}{x_N} \mapsto \sigma\left(\frac{\tau(x_N)}{x_N}\right).$$ 
Ceci est indépendant du choix de $x_N$ car $\mathrm{I}_\ell$ agit trivialement sur les racines $(\ell^N-1)$-ième de l'unité dans $\overline{\mathbb{Q}_{\ell}}$.  
L'image de $\psi_{N}$ est alors le groupe multiplicatif de l'unique sous-corps de $\overline{\mathbb{F}_{\ell}}$  de cardinal $\ell^N$. On peut  vérifier que $\psi_1$ est la réduction modulo $\ell$ du caractère cyclotomique  $\chi_{\ell}$ restreint à $ \mathrm{I}_{\ell}$. D'autre part, si $M$ divise $N$ alors $\psi_M={\psi_N}^{\frac{\ell^N-1}{\ell^M-1}}$.\\

Soit $\psi : {\rm I}_\ell \rightarrow \overline{\mathbb{F}_\ell}^\times$ un caractère continu. Il existe un entier $N\geq 1$ et un entier $0 \leq a < \ell^N$ tel que $\psi = \psi_N^a$. Notons $a'$ l'unique élément de $\{0,\dots,\ell-1\}$ tel que $a \equiv a' \bmod \ell$. Observons que $a'$ ne dépend pas du choix de l'écriture de $\psi$ sous la forme $\psi = \psi_N^a$. En effet, supposons que l'on ait l'égalité $\psi_N^a = \psi_M^b$. On constate que l'on a $\psi_N^a = \psi_{NM}^{a\frac{\ell^{NM}-1}{\ell^N-1}}$, $0 \leq a\frac{\ell^{NM}-1}{\ell^N-1}<\ell^{NM}$ et la congruence $a \equiv a\frac{\ell^{NM}-1}{\ell^N-1} \bmod \ell$. En procédant de manière similaire pour $b$, on peut donc supposer que $N=M$, ce qui entraîne $a=b$. Nous appellerons ${\it poids}$ de $\psi$ l'élément $a' \in \{0,\dots,\ell-1\}$ bien défini par ces observations. On le note $k(\psi)$.

Soit $V$ une $\overline{\mathbb{F}_\ell}$-représentation continue de dimension finie $d$ du groupe $\mathrm {Gal}(\overline{\mathbb{Q}}/\mathbb{Q})$. Soit $U$ la semi-simplifiée de la restriction de $V$ à $\mathrm{I}_\ell$. C'est une somme directe de $d$ caractères. On appelle {\it poids d'inertie modérée} de $V$ la collection des poids de ces caractères. Il s'agit donc d'une famille de $d$ entiers (non nécessairement distincts) de $\{0,\dots,\ell-1\}$.

\begin{lemme}\label{inertielambda}
Soit $V$ une $\overline{\mathbb{F}_\ell}$-représentation continue de dimension finie $d$ du groupe $\mathrm{Gal}(\overline{\mathbb{Q}}/\mathbb{Q})$. On suppose que $V$ a pour poids d'inertie modérée les entiers $k_i$ avec $i=1,\cdots,d$. On suppose de plus que $k_i+k_j <\ell$ pour $i<j$. Alors les poids d'inertie modérée de $\Lambda^2 V$ sont les $k_i+k_j$ avec $i<j$. 
\end{lemme}
La démonstration de ce lemme est laissée au lecteur.

\subsubsection*{Action de l'inertie}
Soit $\ell$ un nombre premier. On a défini dans l'introduction une représentation galoisienne $\ell$-adique $\overline{r_{j,k,\ell}}$ pour $(j,k)\in \{(6,8),(8,8),(4,10), (12,6)\}$. La proposition suivante décrit l'action du groupe d'inertie en $\ell$.
On fixe un plongement $\overline{\mathbb{Q}} \longrightarrow \overline{\mathbb{Q}_\ell}$, ce qui nous fournit un morphisme de groupes $\mathrm{Gal}(\overline{\mathbb{Q}}_\ell/\mathbb{Q}_\ell) \rightarrow \mathrm{Gal}(\overline{\mathbb{Q}}/\mathbb{Q})$. D'après les rappels ci-dessus concernant la définition de $r_{j,k,\ell}$, et d'après \cite{chenevierharris}, la représentation $r_{j,k,\ell}$ est une $\overline{\mathbb{Q}}_\ell$-représentation cristalline de dimension $4$, de poids de Hodge-Tate $0, k-2, j+k-1$, et $j+2k-3$. 
Le théorème de Fontaine-Laffaille \cite{fontainelaffaille} admet donc la conséquence suivante. 
\begin{proposition}\label{description}
 On suppose $\ell > j+2k-2$. Les poids d'inertie modérée de la restriction de $\overline{r_{j,k,\ell}}$ à $\mathrm{Gal}(\overline{\mathbb{Q}_\ell}/\mathbb{Q}_\ell)$, sont $0, k-2, j+k-1$ et $j+2k-3$.
\end{proposition}

\section{Étude de l'irréductibilité}\label{chapitre3}
On va étudier l'irréductibilité des $\overline{r_{j,k,l}}$ pour $(j,k)\in \{(6,8), (8,8), (4,10),$  $(12,6)\}$. Le but sera de démontrer le théorème \ref{irreductible}.
 %
Soient $p$ un nombre premier et $(j,k)$ l'un des quatre couples définis dans l'introduction. Si $r\geq 1$ est un entier, on note $\tau_{j,k}(p^r)$ la trace de $\mathrm{Frob}_p^r$ dans la représentation $r_{j,k,\ell}$. D'après le théorème 1.1 et \cite[Prop. IX.1.9]{chenevier}, c'est un nombre entier indépendant du choix du nombre premier $\ell \neq p$. Le polynôme caractéristique de $r_{j,k,\ell}(\mathrm{Frob}_p)$ est alors de la forme : 
\begin{eqnarray*} &P_{j,k,p}(X)=X^{4}Q_{p}(\frac{1}{X})\\ &=X^{4}-\tau_{j,k}(p)X^3+\frac{{\tau_{j,k}(p)}^{2}-\tau_{j,k}(p^2)}{2}X^2-\tau_{j,k}(p)p^{2k+j-3}X+p^{4k+2j-6}. \end{eqnarray*}

Le calcul des $\tau_{j,k}(p^r)$ pour des petites valeurs de $p$ et $r$ a été effectué par Faber et van der Geer (voir le chapitre 25 dans \cite{123}), et de manière différente par Chenevier et Lannes dans \cite{chenevier}. Nous n'utiliserons ces calculs explicites que pour $p=2, 3, 5,7,13$.

La représentation $\overline{r_{j,k,\ell}}$ vérifie $\overline{r_{j,k,\ell}}^{*}\simeq\overline{r_{j,k,\ell}}\otimes\chi_{l}^{-(j+2k-3)}$. On en déduit que l'ensemble fini $\mathrm{J}(\overline{r_{j,k,\ell}})$ des facteurs de Jordan-H\"older  de $\overline{r_{j,k,\ell}}$ est stable par l'opération $W\mapsto W^{*}\otimes\chi_{l}^{j+2k-3}$. Observons d'abord que $\mathrm{J}(\overline{r_{j,k,\ell}})$ ne peut pas contenir des éléments de dimension $3$. En effet, dans le cas contraire, soit  $W$ un tel élément, alors $W$ est irréductible et la restriction $b_{W}$ de la forme bilinéaire alternée $b$ à $W$ est dégénérée. Son noyau est alors  de dimension $1$ ou $3$. Mais comme $b$ est d'indice $2$, le noyau de $b_{W}$ est de dimension $1$ et c'est une sous-représentation de $W$ ce qui n'est pas possible.\\

L'objectif est de montrer que $\mathrm{J}(\overline{r_{j,k,\ell}})$ est réduit à un singleton sauf dans les cas exclus par le théorème \ref{irreductible}. Si $\ell \leq j+2k-3$, cela résulte de la proposition 4.10 du chapitre X de \cite{chenevier}. Cependant, les critères d'irréductibilité {\it ad hoc} employés {\it loc. cit.}, qui donnent des conditions suffisantes mais non nécessaires a priori, deviennent fastidieux à vérifier, notamment le calcul de l'entier $\kappa$,  quand $\ell$ grandit. De plus, ils ne permettent au mieux que de traiter un nombre fini de valeurs de $\ell$, pour un $(j,k)$ donné. L'ingrédient nouveau utilisé ci-dessous est l'action de l'inertie en $\ell$.\\

On suppose dorénavant $\ell>j+2k-3$. Si $\mathrm{J}(\overline{r_{j,k,\ell}})$ n'est pas un singleton, alors on n'est dans l'une des situations suivantes:

\subsubsection*{Existence d'une droite stable}

Comme la représentation $\overline{r_{j,k,\ell}}$ est non ramifiée hors de $\ell$, l'action de $\mathrm{Gal}(\overline{\mathbb{Q}}/\mathbb{Q})$ sur une droite   stable $D$ est donnée par une puissance entière de $\chi_\ell$ (Kronecker-Weber). Comme de plus on a $\ell > j+2k-2$, la proposition \ref{description} montre que c'est $\chi_\ell^i$ avec $i=0, k-2, j+k-1, j+2k-3$. 
Il s'ensuit que $\chi_{\ell}^{i}(\mathrm{Frob}_{p}) $  pour $\ell$ premier différent de $p$,  est racine du polynôme $P_{j,k,p}$ modulo $\ell$. D'autre part, $D^{*}\otimes\chi_\ell^{j+2k-3}$ est également une droite stable, donc $\chi_{\ell}^{j+2k-3-i}(\mathrm{Frob}_{p}) $  est également racine du polynôme $P_{j,k,p}$ modulo $\ell$.  On a donc, selon la valeur de $i$, que $1$ ou $\chi_{\ell}^{j+k-1}(\mathrm{Frob}_{p})$ est racine du polynôme $P_{j,k,p}$ modulo $\ell$, pour $\ell$ nombre premier différent de $p$. 
\newline Dans le  premier cas, on obtient la condition que $\ell$  divise le nombre:
 $$ A_{j,k}(p)=\frac{{\tau_{j,k}(p)}^{2}-\tau_{j,k}(p^2)}{2}-\tau_{j,k}(p)(p^{2k+j-3}+1)+p^{4k+2j-6}+1.$$
\newline Dans le second cas, on doit avoir $\ell$ qui divise: $$B_{j,k}(p)=p^{2k-4}(1+p^{2j+2})-\tau_{j,k}(p)p^{k-2}(1+p^{j+1})+\frac{{\tau_{j,k}(p)}^{2}-\tau_{j,k}(p^2)}{2}.$$
\subsubsection*{Existence de deux composantes de dimension $2$ reliées}
C'est le cas où la représentation se décompose en somme de deux représentations galoisiennes irréductibles:
$$\overline{r_{j,k,\ell}}\simeq \pi_{1}\oplus\pi_{2}$$ telle que $\pi_{1}^{*}\otimes \chi_{l}^{j+2k-3}\simeq \pi_{2}$. Cette relation montre que les ensembles $P_1$ et $P_2$, constitués respectivement des poids de l'inertie modérée de $\pi_1$ et $\pi_2$, vérifient $P_2 = -P_1 + j+2k-3$. Comme  $\ell > j+2k-2$,  $0$ est dans $P_1$ (quitte à échanger $\pi_{1}$ et $\pi_2$), auquel cas $j+2k-3$ est dans $P_2$, et donc $P_1=\{0,j+k-1\}$ ou $P_1=\{0,k-2\}$ (noter que les quatre poids sont distincts par hypothèses). Deux cas sont alors à envisager.

Dans le premier cas, le polynôme $P_{j,k,p}$  se factorise modulo $\ell$ sous la forme : 
$$P_{j,k,p}(X)=\left(X^2-AX+p^{3k+j-5}\right)\left(X^2-\frac{A}{p^{k-2}}X	+p^{j+k-1}\right).$$
En identifiant et en éliminant $A$, on voit que $\ell$ doit diviser
$$C_{j,k}(p)=\left(\frac{{\tau_{j,k}(p)}^{2}-\tau_{j,k}(p^2)}{2}-p^{j+k-1}-p^{j+3k-5}\right){\left(1+p^{k-2}\right)}^{2}-p^{k-2}{\tau_{j,k}(p)}^{2}.$$

Dans le second cas, on doit avoir une factorisation de la forme:
$$P_{j,k,p}(X)=\left(X^2-AX+p^{3k+2j-4}\right)\left(X^2-\frac{A}{p^{j+k-1}}X+p^{k-2}\right).$$ 
En identifiant et en éliminant $A$, on voit que $\ell$ doit diviser:
$$D_{j,k}(p)=\left(\frac{{\tau_{j,k}(p)}^{2}-\tau_{j,k}(p^2)}{2}-p^{k-2}-p^{2j+3k-4}\right){(1+p^{j+k-1})}^{2}-p^{j+k-1}{\tau_{j,k}(p)}^{2}.$$

\subsubsection*{Existence de deux composantes de dimension $2$ non reliées}
Dans ce dernier cas, $\overline{r_{j,k,\ell}}$ est la  somme de deux représentations irréductibles ${\overline{r_{j,k,l}}}\simeq \pi_{1}\oplus\pi_{2}$ avec $\pi_{1}^{*}\otimes \chi_\ell^{j+2k-3}\simeq \pi_{1}$ et $\pi_{2}^{*}\otimes \chi_\ell^{j+2k-3}\simeq \pi_{2}$. 

Comme $\ell > j+2k-2$, la proposition \ref{description} s'applique. 
Quitte à échanger les rôles de $\pi_1$ et $\pi_2$, on peut supposer que $0$ est un poids de l'inertie modérée de $\pi_1$, de sorte que ses deux poids sont $0$ et $j+2k-3$, et ceux de $\pi_2$ sont $k-2$ et $j+k-1$.

On utilise alors la résolution par Wintenberger et Khare de la conjecture de Serre (voir \cite{khare} et \cite{serre1987}). En effet, considérons la représentation $\rho:=\pi_2 \otimes \chi_\ell^{2-k}$. Elle est irréductible (comme représentation du groupe de Galois absolu de $\mathbb{Q}$). Elle est impaire car son déterminant est $\chi_\ell^{j+2k-3-2k+4}=\chi_\ell^{j+1}$, et $j+1$ est impair. Comme elle est non ramifiée hors de $\ell$, le conducteur $\mathrm{N}(\rho)$ de $\rho$ est $1$. Les poids de l'inertie modérée de $\rho$ sont $0< j+1 < \ell$. Si la restriction de $\rho$ \`a $\mathrm{I}_\ell$ est irréductible, alors le poids de Serre $k(\rho)$ vaut $j+2$ par la recette de Serre. Dans les autres cas, elle est réductible de semi-simplifiée $1 \oplus \chi_\ell^{j+1}$. Mais par construction de $\overline{r_{j,k,\ell}}$, la restriction à $\mathrm{Gal}(\overline{\mathbb{Q}_\ell}/\mathbb{Q}_\ell)$ de la représentation $\pi_2$ est du type de celles étudiées par Fontaine et Laffaille dans \cite{fontainelaffaille}. On a $k-2<j+k-1$, d'après ces auteurs (voir aussi le théorème 2 de \cite{wach}), c'est donc nécessairement une extension de $\chi_\ell^{k-2}$ par $\chi_\ell^{j+k-1}$, comme on le voit sur le module filtré associé. Ainsi, $\rho$ est une extension de $1$ par $\chi_\ell^{j+1}$, de sorte que la recette de Serre donne encore $k(\rho)=j+2$ car on a $1<j+1<\ell-1$. D'après Khare, $\rho$ est donc la représentation résiduelle associée à une forme modulaire parabolique pour $\mathrm{SL}_2(\mathbb{Z})$ de poids $j+2$. Mais pour $j=4,6,8,12$, l'espace des formes modulaires paraboliques de poids $j+2$ pour $\mathrm{SL}_2(\mathbb{Z})$ est nul, une contradiction.

\subsubsection*{Résultats des calculs explicites}
 En utilisant les tables données dans \cite{chenevier}, on calcule $A_{j,k}(p)$, $B_{j,k}(p)$, $C_{j,k}(p)$ et $D_{j,k}(p)$ pour certains nombres premiers $p$. On retiendra le fait que les conditions de réductibilité précédemment étudiées imposent à $\ell$ de diviser les éléments de chaque ligne des tables $2$, $3$, $4$ et $5$, dès que  $\ell$ est supérieur à $j+2k-2$. On remarquera par exemple que $23$ divise les  éléments de la troisième ligne pour $(j,k)=(8,8)$ et $41$ divise les éléments de la deuxième ligne pour $(j,k)=(4,10)$. Cela permet donc de compléter la démonstration du théorème \ref{irreductible}.\\
 
Le résultat suivant résulte sans doute (pour tout $\ell$) de la construction des $r_{j,k;\ell}$ par Weissauer \cite{weissaur}. Nous en proposons ci-dessous une démonstration alternative.  
\begin{corollaire} \label{corollaireirréductible}
Soient $(j,k)$ l'un des quatre couples ci-dessus et $\ell$ un nombre premier $\geq 7$. Alors la représentation $r_{j,k,\ell}$ est définie sur $\mathbb{Q}_\ell$. En particulier, son image est conjuguée à un sous-groupe de $\mathrm{GSp}_4(\mathbb{Z}_\ell)$.
\end{corollaire}
\begin{proof}[Preuve du corollaire \ref{corollaireirréductible}] 	
Nous avons déjà dit que la représentation $r_{j,k,\ell}$ est irréductible, et à traces dans $\mathbb{Q}(f)_\ell=\mathbb{Q}_\ell$. Soit $R$ la $\mathbb{Q}_\ell$-algèbre engendrée par le groupe $r_{j,k,\ell}(\mathrm{Gal}(\overline{\mathbb{Q}}/\mathbb{Q})$ dans l'anneau des endomorphismes de $V$. La théorie de Wedderburn assure que $R$ est centrale et simple de rang $16$. D'après la propriété (ii) du théorème \ref{chenevier1}, les conjugaisons complexes de $\mathrm{Gal}(\overline{\mathbb{Q}}/\mathbb{Q})$ ont pour valeurs propres $1,1,-1,-1$. Ce n'est donc pas une algèbre à division et on a alors deux cas possibles : soit  $R \simeq \mathrm{M}_{4}(\mathbb{Q}_\ell)$, ce qui équivaut à dire que la représentation $r_{j,k,\ell}$ est définie sur $\mathbb{Q}_\ell$, soit  $R \simeq \mathrm{M}_2(H)$ où $H$ est l'unique algèbre de quaternions sur $\mathbb{Q}_\ell$. Dans ce deuxième cas, la représentation résiduelle $\overline{r_{j,k,\ell}}$ est alors somme directe de deux représentations de dimension $2$, disons $(V_1,r_1)$ et $(V_{2},r_2)$, toutes les deux définies sur $\mathbb{F}_{\ell^2}$, et telles que $(V_{2},r_2)\simeq (V_{1}, \phi_{\ell}\circ r_1)$, où $\phi_{\ell}: \mathrm{GL}(V_{1})\rightarrow\mathrm{GL}(V_{1})$ est l'élévation à la puissance $\ell$. Bien entendu, cette possibilité est exclue si $\overline{r_{j,k,\ell}}$ est irréductible, de sorte qu'il n'y a qu'un nombre fini de cas restant à traiter d'après le théorème précédent. Or, dans ces cas, la représentation $\overline{r_{j,k,\ell}}$ a été étudiée dans le Théorème X.4.4 de \cite{chenevier}. D'après cette référence, elle est alors réductible, tous ses facteurs sont définis sur $\mathbb{F}_\ell$, et ils sont sans multiplicité si $\ell\geq 7$. 	
\end{proof}

\section{Sous-groupes maximaux du groupe projectif symplectique}
Soient $k$ un corps de caractéristique différente de $2$, $V$ un espace vectoriel de dimension $4$ sur $k$ muni  d'une forme bilinéaire alternée non dégénérée $A$. La forme $A$ définit  un élément de $\Lambda^{2}V^{*}$.  On définit le groupe des similitudes symplectiques de $A$ comme suit:
$$\mathrm{GSp}(A):=\{g\in \mathrm{GL}(V), \exists \eta(g) \in k^{\times}, \forall x,y \in V, A(g(x),g(y))=\eta(g) A(x,y) \}.$$ 
Le scalaire $\eta(g)$ s'appelle facteur de similitude de $g$ et l'application $\eta : \mathrm{GSp}(A)\rightarrow k^{\times}$ est un morphisme de groupe, de noyau le groupe $\mathrm{Sp}(A)$, le groupe symplectique associé à $A$, et on a la suite exacte courte:
$$1\rightarrow \mathrm{Sp}(A)\rightarrow \mathrm{GSp}(A)\stackrel{\eta}{\rightarrow} k^{\times}\rightarrow 1.$$

On suppose que $V$ est muni d'une base $(e_{0},e_{1},e_{2},e_{3})$ dans laquelle la matrice de $A$ est $J=\begin{pmatrix} 0 & I_{2} \\ -I_{2} & 0 \\ \end{pmatrix}$.   
On note dans ce cas    $$\mathrm{GSp}_{4}(k):= \mathrm{GSp}(A), \quad \mathrm{Sp}_{4}(k):= \mathrm{Sp}(A)$$ et on a donc : $$\mathrm{GSp}_{4}(k)=\{M\in \mathrm{M}_{4}(k), \exists \eta(M) \in k^{\times} : M^{t}JM=\eta(M)J\}.$$
On note  $\mathrm{PGSp}(A)$  le quotient du groupe $\mathrm{GSp}(A)$ par son centre, formé des homothéties. De même, $\mathrm{PSp}(A)$ est le quotient de $\mathrm{Sp}(A)$ par $\{\pm \mathrm{Id}_{V}\}$, où $\mathrm{Id}_{V}$ désigne l'application identité de $V$. On a alors la suite exacte:
$$1\rightarrow \mathrm{PSp}(A)\rightarrow \mathrm{PGSp}(A)\stackrel{\eta}{\rightarrow} k^{\times}/{k^{\times}}^{2}\rightarrow 1.$$
On dispose d'une application bilinéaire : 
$$b: \Lambda ^{2}V \times \Lambda ^{2}V\rightarrow \Lambda ^{4}V$$
induite par l'application multilinéaire alternée \begin{eqnarray*} &V\times V\times V\times V &\rightarrow \Lambda ^{4}V\\ &(u,v,w,z) &\mapsto u\wedge v\wedge w\wedge z. \end{eqnarray*} 
L'application $b$ est symétrique car les tenseurs d'ordre $2$ commutent. Elle est non dégénérée. Ceci peut se voir simplement en écrivant sa matrice dans la base des tenseurs purs associés à la base $(e_{0},e_{1},e_{2},e_{3})$. Identifions (de manière un peu arbitraire) la droite $\Lambda^4 V$ à $k$ au moyen de la base $e_0 \wedge e_1 \wedge e_2 \wedge e_3$. Cela permet de voir $b$ comme une forme bilinéaire symétrique non dégénérée sur $\Lambda^2 V$ (on peut même vérifier qu'elle est hyperbolique, ou autrement dit, qu'elle possède un sous-espace isotrope de dimension $3$). En particulier, elle induit un isomorphisme entre $\Lambda^2 V$ et $\Lambda^2 V^\ast$. Notons $\widetilde{A}$ l'élément de $\Lambda^2 V$ qui correspond à la forme symplectique $A$ via cet isomorphisme. Par définition, on a la relation $A(v,w) e_0 \wedge e_1 \wedge e_2 \wedge e_3 = v \wedge w \wedge \widetilde{A}$ pour tous $v,w \in V$, et donc l'identité $\widetilde{A} = - e_0 \wedge e_2 - e_1 \wedge e_3$. En particulier, on constate que l'on a $b(\widetilde{A},\widetilde{A}) = -2 \neq 0$ : l'élément $\widetilde{A}$ n'est pas isotrope pour $b$ dans $\Lambda^2 V$.\\

Le groupe linéaire de $V$ agit de manière naturelle sur $\Lambda ^{2}V$. Tout élément de $\mathrm{GL}(V)$ est une similitude pour $b$, c'est-dire un élément de $\mathrm{GO}(b)$. De plus, le facteur de similitude d'un élément $g\in\mathrm{GL}(V)$ pour la forme $b$ est le déterminant de $g$. En effet, cela découle de l'égalité,  valable pour $u,v,w,z$ des éléments de $V$:  

$$(\wedge^{4}g)\,(u\wedge v\wedge w\wedge z)=\det(g)\, u\wedge v\wedge w\wedge z.$$

Le sous-groupe $\mathrm{GSp}(A)$ agit, par définition, de manière scalaire sur $A$, et par suite sur $\widetilde{A}$ aussi. Il stabilise donc  $H=\widetilde{A}^{\bot}$. La forme $b$ restreinte \`a $H$, que l'on notera $b_H$, est non-dég\'en\'er\'ee car on a vu que $\widetilde{A}$ n'est pas isotrope. On a ainsi défini un morphisme naturel $\iota : \mathrm{GSp}(A) \rightarrow \mathrm{GO}(b_H)$. En fait, le groupe $\mathrm{GSp}(A)$ agit sur $H \otimes \eta^{-1}$ par des éléments de $\mathrm{SO}(b_H)$, de sorte que l'on a construit un morphisme de groupes $\mathrm{PGSp}(A) \rightarrow \mathrm{SO}(b_H)$ ; il est bien connu que c'est un isomorphisme.\\

On suppose que $\mathrm{char}\,k \neq 3$. On rappelle que le groupe $\mathrm{GL}_2(k)$ agit sur l'espace $\mathrm{Sym}^3 k^2$ par des similitudes symplectiques, relativement \`a une forme bilinéaire alternée non dégénérée convenablement fixée sur $\mathrm {Sym}^3 k^2$ (une telle forme étant d'ailleurs unique à un scalaire près). Cela définit un homomorphisme injectif $\mathrm{PGL}_2(k) \rightarrow \mathrm{PGSp}_4(k)$, appelé {\it plongement principal}, dont la classe de conjugaison est canonique. Un tel sous-groupe de $\mathrm{PGSp}_4(k)$ peut aussi être défini comme étant le stabilisateur d'une ``cubique tordue'' de l'espace projectif $\mathrm{P}(V)$ \cite{mitchell}.\\

Mitchell a classifié dans \cite{mitchell} les sous-groupes maximaux de $\mathrm{PSp}_{4}(\mathbb{F}_{q})$. Cette classification est reprise par King dans \cite{king}, page $8$. On peut vérifier que ces diverses formulations impliquent la classification suivante pour $\mathrm{PGSp}_{4}(\mathbb{F}_{q})$. 

\begin{theo}\label{classification2} Soient $\ell$ un nombre premier $>3$, $V$ et $H$ les $\mathbb{F}_\ell$-représentations naturelles de $\mathrm{ GSp}_4(\mathbb{F}_\ell)$ décrites ci-dessus, de dimensions respectives $4$ et $5$. Soit $G$ un sous-groupe de $\mathrm{PGSp}_4(\mathbb{F}_\ell)$. 
Alors soit $G$ contient $\mathrm{PSp}_4(\mathbb{F}_\ell)$, soit il est inclus dans l'un des sous-groupes suivants :\\ 
\begin{itemize} 
\item[(1)] Stabilisateur d'une droite de $V$, \\
\item[(2)] Stabilisateur d'un plan isotrope de $V$ (ou ce qui revient au même, d'une droite isotrope de $H$), \\
\item[(3)] Stabilisateur d'une droite non isotrope de $H$, \\
\item[(4)] Stabilisateur d'un plan non dégénéré de $H$, \\
\item[(5)] Le groupe $\mathrm{PGL}_2(\mathbb{F}_\ell)$ plongé de manière principale, \\
\item[(6)] Un groupe contenant un sous-groupe d'indice $\leq 2$ qui est extension de $\mathcal{A}_5$, ou de $\mathcal{S}_5$, par $(\mathbb{Z}/2\mathbb{Z})^4$, \\
\item[(7)] Un groupe contenant un sous-groupe d'indice $\leq 2$ isomorphe \`a $\mathcal{A}_6, \mathcal{S}_6$ ou $\mathcal{A}_7$ (et dans ce dernier cas on a $\ell=7$). \\
\end{itemize}
\end{theo}
\begin{remarque}\rm Dans la terminologie de Mitchell, si $U$ est un sous-groupe de $\mathrm{PSp}_4(\mathbb{F}_\ell)$ du type (2), (3), (4) ou (5), alors $U$ est le stabilisateur d'une congruence parabolique, d'une congruence hyperbolique ou elliptique de $\Lambda^{2}V$ conte\-nant $\widetilde{A}$, d'une quadrique de $\Lambda^{2}V$ ou d'une cubique tordue de $\Lambda^{2}V$.
\end{remarque}

\section{Étude de l'image}

Notre objectif dans ce chapitre est de démontrer le théorème \ref{images} de l'introduction. Soient $(j,k)$ l'un des quatre couples de l'introduction, $\ell$ un nombre premier, $\rho$ la représentation de $\mathrm{Gal}(\overline{\mathbb{Q}}/\mathbb{Q})$ notée $r_{j,k,\ell}$ dans l'introduction, et $\overline{\rho}:=\overline{r_{j,k,\ell}}$ sa représentation résiduelle. Comme on l'a déjà expliqué, $\overline{\rho}$ est définie sur $\mathbb{F}_\ell$, et prend ses valeurs dans le groupe des similitudes symplectiques d'une forme bilinéaire alternée non dégénérée bien choisie, de sorte que l'on peut supposer que l'on a  $$\mathrm{Im} \, \overline{\rho} \, \subset \mathrm{GSp}_4(\mathbb{F}_\ell).$$ 
Conformément à la section $4$, on notera $V$ le $\mathbb{F}_\ell$-espace vectoriel de dimension $4$ sous-jacent, muni d'une forme bilinéaire alternée non dégénérée. 
On note de plus $G_\ell$ l'image de $\overline{\rho}$ dans $\mathrm{GSp}_4(\mathbb{F}_\ell)$ et $PG_\ell$ l'image de $G_\ell$ dans $\mathrm{PGSp}_4(\mathbb{F}_\ell)$.\\

On se place dans   les hypothèses du théorème \ref{irreductible}, de sorte que la représentation $\overline{\rho} = V \otimes \overline{\mathbb{F}_\ell}$ est irréductible. En particulier, on a $\ell\geq 7$. On suppose que $PG_\ell$ ne contient pas $\mathrm{PSp}_4(\mathbb{F}_\ell)$. Il est donc contenu dans l'un des sous-groupes maximaux du type (1) \`a (7) dans la liste du théorème \ref{classification2}. Nous allons procéder à une étude au cas par cas. L'idée est de montrer que la forme de l'image est incompatible soit avec l'action de l'inertie quand $\ell$ est assez grand, soit avec l'action du Frobenius en $p\neq\ell$ quand $\ell$ est petit. On rappelle que la représentation $H \subset \Lambda^2 V$ a été introduite à la section $4$ et que $\Lambda^2 V$ est muni d'une forme bilinéaire symétrique non dégénérée naturelle. \\

%
%

  %
%

{\it Cas (1) ou (2) : Stabilisateur d'une droite, ou d'un plan isotrope, de $V$.} Ce n'est pas possible car  la représentation est irréductible.  \\

{\it Cas (3) ou (4) : Stabilisateur d'une droite ou d'un plan non dégénérés de $H$.} Dans les deux cas, $G_\ell$ préserve un plan non dégénéré $P \subset \Lambda^2 V$. En particulier, la représentation naturelle de $\mathrm{ Gal}(\overline{\mathbb{Q}}/\mathbb{Q})$ sur $P \otimes \chi_\ell^{-w}$, avec  $w=j+2k-3$, définit un homomorphisme 
		$$r: \mathrm{ Gal}(\overline{\mathbb{Q}}/\mathbb{Q})\longrightarrow \mathrm{ O}(P)$$
(le groupe orthogonal du plan quadratique $P$). Il est bien connu que les vecteurs isotropes de $\Lambda^2 V \otimes \overline{\mathbb{F}}_\ell$ sont exactement les tenseurs purs ("quadrique de Klein"). Comme $\rho$ est irréductible, elle n'admet pas de plan stable, et donc $\mathrm{ Gal}(\overline{\mathbb{Q}}/\mathbb{Q})$ ne stabilise aucune des deux droites isotropes de $P \otimes \overline{\mathbb{F}_\ell}$. Son image n'est donc pas incluse dans $\mathrm{ SO}(P)$. Autrement dit, l'homomorphisme 
$$\epsilon := \det r : \mathrm{ Gal}(\overline{\mathbb{Q}}/\mathbb{Q}) \rightarrow \{\pm 1\}$$
est non trivial. Il est continu, et non ramifié hors de $\ell$, par construction. Il définit donc une extension quadratique de $\mathbb{Q}$ non ramifiée hors de $\ell$. Mais il est bien connu qu'une telle extension est nécessairement ramifiée en $\ell$ (c'est $\mathbb{Q}\left(\sqrt{(-1)^\frac{\ell-1}{2}\ell}\right)$), autrement dit que $\epsilon(\mathrm{ I}_\ell)=\{\pm 1\}$. On a donc montré que 
\begin{equation}
\label{epsilon}
\epsilon=\chi_{\ell}^\frac{\ell-1}{2}.
\end{equation}
D'autre part, on a $\ell>2$ donc tout élément de $\mathrm{O}(P)$ est semi-simple. L'inertie sauvage de $\mathrm{I}_\ell$ agit trivialement dans $r$, et le sous-groupe $r (\mathrm{ I}_\ell)$ est cyclique (d'ordre premier à $\ell$). Si $\sigma$ en est un générateur, on a donc $\det(\sigma)=-1$ : c'est une symétrie orthogonale. Ainsi, la restriction de $r$ à $\mathrm{ I}_\ell$ est somme du caractère trivial et d'un caractère d'ordre $2$ : c'est donc $1 \oplus \chi_\ell^{\frac{\ell-1}{2}}$.\\
 
Pour éliminer cette éventualité, faisons d'abord l'hypothèse supplémentaire $$\ell>2w,$$
avec $w=j+2k-3$.
Sous cette hypothèse, qui entraîne $\ell> j+2k-2$ et $2j+3k-4 \leq \ell -1$, la proposition \ref{description} et le lemme \ref{inertielambda} assurent que les poids d'inertie modérée agissant sur $\Lambda^2 V$ sont les entiers distincts 
\begin{equation}  k-2 <j+k-1 < j+2k-3 < j+3k-5 < 2j+3k-4\label{listepoidsl2} \end{equation}
(on a dans tous les cas $j \geq 4$, $k \geq 6$), avec $j+2k-3$ de multiplicité $2$. En particulier, les deux poids de l'inertie modérée de $r \otimes \chi_\ell^w \subset \Lambda^2 V$ sont dans cette liste. Mais d'après le paragraphe précédent, et l'hypothèse $\ell > 2w$, ces deux poids sont $w$ et $w+\frac{\ell-1}{2}$. En particulier, $\frac{\ell-1}{2}+w$ est dans la liste \eqref{listepoidsl2}, ce qui est absurde par l'inégalité $\frac{\ell-1}{2}+w \geq 2w > 2j+3k-4$.\\

Expliquons maintenant comment éliminer les cas restants, pour lesquels on a $7 \leq \ell \leq 41$. Regardons l'action de $\mathrm{Frob}_p$ sur $\Lambda^2 V \otimes \chi_\ell^{-w}$ (avec $p \neq \ell$). C'est un élément de $\mathrm{SO}(H)$ dont le polynôme caractéristique est la réduction modulo $\ell$ de 
		$$R_p(t)=(t-1)(t^4- a t^3 + b t^2 -a t + 1)$$
avec $a=\left( \frac{\tau_{j,k}(p)^2 - \tau_{j,k}(p^2)}{2 \, p^{w}} - 2\right)$ et $b=\tau_{j,k}(p^2)/p^w +2$, comme on le vérifie facilement. Donc, si $p$ est un nombre premier qui n'est pas un carré modulo $\ell$, alors $r(\mathrm{Frob}_p) \in \mathrm{O}(P)-\mathrm{SO}(P)$ d'après (\ref{epsilon}) et donc $-1$ est valeur propre, i.e $R_p(-1)=0$. On peut alors procéder de la manière suivante. On constate d'abord que l'un des premiers $2$, $3$, $5$ est toujours un non-carré modulo $\ell$ lorsque $7 \leq \ell \leq 41$. On vérifie dans tous les cas que si $\overline{r_{j,k,\ell}}$ est irréductible (i.e. que le triplet $(j,k,\ell)$ est dans la liste du théorème \ref{irreductible}), et si $(j,k,\ell)$ n'est ni $(6,8,13)$, ni $(4,10,17)$, alors il existe toujours $p \in \{2, 3, 5\}$ non carré modulo $\ell$ et tel que $R_p(-1) \neq 0$. Traitons par exemple le cas $(j,k)=(6,8)$. Le cas $\ell=13$ est exclus par l'énoncé.  De plus, on vérifie que l'on a $R_2(-1)=0$, $R_3(-1)= -2^7 \cdot 5^6 / 3^{13}$ et $R_5(-1)=-2^5\cdot 3^{14} \cdot 13^2 \cdot 5^{-15}$. On a de plus $\ell \neq 11$ car $\overline{r_{6,8,11}}$ est réductible. On conclut car soit $3$, soit $5$, est non-carré modulo $\ell$.\\

{\it Cas (5): Stabilisateur d'une cubique de $\mathrm{P}(V)$. } Dans ce cas, pour tout élément $\sigma$ de $G_\ell$, il existe $x,y,z \in \overline{\mathbb{F}_\ell}^\ast$ tels que $\sigma$ admet pour valeurs propres dans $V \otimes  \overline{\mathbb{F}_\ell}$ les éléments $x^a y^{3-a} z$, avec $0 \leq a \leq 3$. La restriction de $\overline{\rho}$ à $\mathrm{I}_\ell$ se factorise par un quotient fini $K$ de ce dernier, qui est produit semi-direct d'un groupe cyclique $\langle \tau \rangle$ d'ordre premier \`a $\ell$ par un $\ell$-groupe fini (inertie sauvage). Appliquons l'observation précédente à l’élément $\sigma = \overline{\rho}(\tau)$. On sait que la semi-simplifiée de $\overline{\rho}_{|\mathrm{I}_\ell}$ est somme de quatre caractères $\chi_1,\dots,\chi_4$. Les valeurs propres précédentes s'écrivent donc également $\chi_i(\tau)$, $i=1,\cdots, 4$. On en déduit que quitte à renuméroter les $\chi_i$, on peut supposer que l'on
 a les égalités 
$$ \chi_1(\tau) / \chi_2(\tau) = \chi_2(\tau)/\chi_3(\tau) = \chi_3(\tau)/\chi_4(\tau).$$
Autrement dit, on a $\chi_1 \chi_3 = \chi_2^2$ et $\chi_2 \chi_4 = \chi_3^2$. Un argument similaire à celui du lemme \ref{inertielambda} montre que sous l'hypothèse $\ell > 2w$, l'ensemble des poids de l'inertie modérée de $V$ est de la forme $\{k_1,k_2,k_3,k_4\}$ avec 
$$k_1+k_3 = 2k_2 \, \, \, {\rm et}\, \, \, \, k_2+k_4=2k_3.$$
On en déduit que $k_i$ et $k_{i+2}$ ont même parité. C'est impossible, car dans la liste $\{k_1,k_2,k_3,k_4\}=\{0,k-2,j+k-1,j+2k-3\}$, les éléments de même parité sont consécutifs (noter que l'on a $j>2$).\\

Enfin, pour traiter les cas restants, à savoir $7 \leq \ell \leq 41$, on observe que si $p \neq \ell$, les racines de $R_p(t)$ dans $\overline{\mathbb{F}_\ell}$ sont de la forme $1,u^2,u,u^{-1},u^{-2}$ avec $u \in \overline{\mathbb{F}_\ell}^\times$. En effet, l'action de $\mathrm{Frob}_p$ sur $V$ a des valeurs propres de la forme $x^ay^{3-a}z$ avec $0\leq a\leq 3$, $z\in \overline{\mathbb{F}_\ell}^\times$ et $x^3y^3z^2=p^{j+2k-3}=\chi_\ell^{w}(\mathrm{Frob}_p)$. Son action sur $\Lambda^2 V$ a donc pour valeurs propres $x^5yz^2,x^4y^2z^2,x^3y^3z^2,x^3y^3z^2,x^2y^4z^2,xy^5z^2$, d'où l'on déduit  les valeurs propres de l'action sur $\Lambda^2 V \otimes \chi_\ell^{-w}$ en divisant par $x^3y^3z^2$ et en posant $u=xy^{-1}$.   
Posons $Q_p(t) = R_p(t)/(t-1)$, alors les images dans $\mathbb{F}_{\ell}[t]$ de $Q_p(t)$ et $Q_p(t^2)$ ne sont pas premières entre elles. On vérifie à l'aide du logiciel \texttt{PARI} \cite{pari} que si $\overline{r_{j,k,\ell}}$ est irréductible, et si $7 \leq \ell \leq 41$, alors il existe toujours un nombre premier $p \in \{2, 3, 5\}$ tel que cette propriété est mise en défaut. Par exemple si $(j,k)=(8,8)$, auquel cas $\ell = 13$ et $17$ sont exclus (cas réductibles), le nombre premier $p=2$ convient toujours.\\

{\it Cas (6) et (7).} Supposons $\ell \geq 7$. Dans ces cas, on observe que tout élément de $PG_\ell$ est d'ordre divisant $16$, $20$,$24$ ou, si $\ell=7$, $14$. Supposons d'abord $\ell >7$, alors $\ell$ ne divise pas l'ordre de $G_\ell$, de sorte que l'inertie sauvage en $\ell$ agit trivialement dans $V$. Le groupe $\mathrm{I}_\ell$ agit donc de manière diagonalisable sur $V \otimes \overline{\mathbb{F}_\ell}$, à travers un quotient cyclique d'ordre premier à $\ell$. Soient $\sigma \in G_\ell$ un générateur de l'image de $\mathrm{I}_\ell$ et $N\geq 1$ l'ordre de l'image de $\sigma$ dans $PG_\ell$. D'après l'analyse ci-dessus, l'entier $N$ divise $16,20$ ou $24$. En particulier $N \leq 24$. De plus, $\sigma^N$ agit par une homothétie sur $V$. Ainsi, si $\chi_1,\cdots,\chi_4$ désignent les caractères de $\mathrm{I}_\ell$ intervenant dans $V \otimes \overline{\mathbb{F}_\ell}$, alors les $\chi_i^N$, $i=1,\dots,4$, sont tous égaux. Supposons de plus que l'on a 
$$ \ell -1 \geq 24 w,$$
où $w=j+2k-3$. Pour chaque entier $1 \leq i \leq 4$, écrivons $\chi_i$ sous la forme $\psi_M^{a_0 + \ell a_1 + \dots +\ell^{M-1} a_{M-1}}$ avec $M\geq 1$ convenable et $0 \leq a_s \leq \ell-1$ pour tout $s$. On sait que la collection des $a_0$ est exactement l'ensemble des poids de l'inertie modérée de $V$ qui sont donnés par la proposition \ref{description}. Sous l'hypothèse $\ell -1 \geq 24 w$, les classe modulo $\ell$ de ces poids multipliés par $N$ restent distinctes. C'est bien entendu absurde. \\

Cette analyse montre que si $\ell > 24\cdot 21=504$, alors on n'est jamais dans les cas (6) ou (7). Pour conclure, expliquons comment montrer que ces cas ne sont en fait pas possibles pour tout $\ell \geq 7$. Il n'y a qu'un nombre fini de cas à écarter. On procède de la manière suivante. Regardons à nouveau l'action de $\mathrm{Frob}_p$ sur $\Lambda^2 V \otimes \chi_\ell^{-w}$ (avec $p \neq \ell$). C'est un élément de $\mathrm{SO}(H)$ dont le polynôme caractéristique est la réduction modulo $\ell$ du polynôme  $R_p(t)$ déjà introduit plus haut. Soit $Q$ le polynôme $(t^{16}-1)(t^{20}-1)(t^{24}-1)$. On vérifie, par exemple en utilisant le logiciel \texttt{PARI} \cite{pari}, que pour chacun des $4$ couples $(j,k)$, et pour tout $7 \leq \ell <504$ tel que $(j,k,\ell)$ soit comme dans le théorème \ref{images}, alors $R_2(t)$,  $R_3(t)$, $R_5(t)$ ou $R_{13}(t)$  ne divise pas $Q(t)^4$ dans $\mathbb{F}_\ell[t]$. Par exemple, pour $(j,k)=(6,8)$, et pour $7 \leq \ell <504$, alors $R_2(t)$ (resp. $R_3(t)$, resp. $R_5(t)$) divise $Q(t)^4$ dans $\mathbb{F}_\ell[t]$ si, et seulement si $\ell = 7,11,13,17,23,47, 103, 107, 109, 461$ (resp. $\ell = 7,11,17,23,359$, resp. $\ell=7,11,31$). Pour éliminer le cas $\ell=7$, on constate que $R_{13}(t)$ ne divise pas $Q(t)^4$ dans $\mathbb{F}_7[t]$.\\

%

%
%
%

{\it L’image dans $\mathrm{GSp}_4(\mathbb{Z}_\ell)$.} La propriété (ii) de la représentation $r_{j,k,\ell}$, et la surjectivité du caractère cyclotomique $\chi_\ell : \mathrm{Gal}(\overline{\mathbb{Q}_\ell}/\mathbb{Q}_\ell) \rightarrow \mathbb{Z}_\ell^\times$, montrent que $\nu(G_\ell)$ est le sous-groupe des puissances $w$-ièmes du groupe $\mathbb{Z}_\ell^\times$, avec $w=j+2k-3$. Il suffit donc de démontrer que si $PG_\ell$ contient $\mathrm{ PSp}_4(\mathbb{F}_\ell)$, et si $\ell > 3$, alors l'image de $\rho$ contient $\mathrm{Sp}_4(\mathbb{Z}_\ell)$.\\

On rappelle que si $\ell >2$, le seul sous-groupe distingué non trivial de $\mathrm{ Sp}_4(\mathbb{F}_\ell)$ est son centre $\{\pm 1\}$. Cela entraîne que si un sous-groupe $H$ de $\mathrm{ Sp}_4(\mathbb{F}_\ell)$ s'envoie surjectivement sur $\mathrm{ PSp}_4(\mathbb{F}_\ell)$, et si $\ell >2$, alors $H = \mathrm{ Sp}_4(\mathbb{F}_\ell)$ (sinon $H$ serait un sous-groupe d'indice $2$, nécessairement distingué, de  $\mathrm{ Sp}_4(\mathbb{F}_\ell)$). Ceci étant dit, supposons que $PG_\ell$ contienne $\mathrm{ PSp}_4(\mathbb{F}_\ell)$. Notons $\mathrm{D}(K)$ le groupe dérivé du groupe $K$. On constate que $\mathrm{D}(G_\ell)$ (qui est un sous-groupe de $\mathrm{Sp}_4(\mathbb{F}_\ell)$) s'envoie surjectivement sur $\mathrm{D}(\mathrm{PSp}_4(\mathbb{F}_\ell))=\mathrm{PSp}_4(\mathbb{F}_\ell)$. La remarque précédente montre donc que $\mathrm{Sp}_4(\mathbb{F}_\ell) = \mathrm{ D}(G_\ell) \subset G_\ell$. On conclut par le lemme suivant, dû à Serre \cite[page 52]{serreo}.

\begin{lemme}\label{serre}
 Soient $\ell$ un nombre premier $>3$ et $G$ un sous-groupe fermé de $\mathrm{GSp}_4(\mathbb{Z}_\ell)$ tel que son image dans $\mathrm{GSp}_4(\mathbb{F}_\ell)$ contienne $\mathrm{Sp}_4(\mathbb{F}_\ell)$. Alors $G$ contient $\mathrm{Sp}_4(\mathbb{Z}_\ell)$.
\end{lemme}

%

\newpage

\section{Calculs explicites}

%
%
%
%
%
%

 {\renewcommand{\arraystretch}{1.5}

\begin{table}[h]
\begin{center}
\small
\begin{tabular}{|c|c|c|c|c|}

\hline
$q$ &$\tau_{6,8}(q) $& $\tau_{8,8}(q) $ \\

\hline
$2$ & $0$ & $1344$  \\
\hline
$3$ & $-27000$ & $-6408$   \\
\hline
$4$ & $409600$& $348160$   \\
\hline
$9$ & $333371700$&$748312020$ \\
\hline
$5$ & $2843100$&$-30774900$ \\
\hline
$25$ & $-15923680827500$&$-395299890927500$ \\
\hline
$7$ & $-107822000$&$451366384$ \\
\hline
$49$ & $-253514141409500$&$-155544419215478300$\\
\hline
$13$ & $9952079500$&$-328006712228$ \\
\hline
$169$ & $-4843967045593944889100$&$-596184280686941758305260$ \\
\hline
\hline
$q$ &  $\tau_{12,6}(q)$& $\tau_{4,10}(q)$ \\

\hline
$2$ &  $-240 $& $-1680$ \\
\hline
$3$ &  $68040$ & $55080$  \\
\hline
$4$ &  $4276480$ & $-700160$ \\
\hline
$9$ & $-8215290540$ & $1854007380 $\\
\hline
$5$ & $14765100$ & $-7338900$\\
\hline
$25$ & $722477627072500$ & $-904546757727500 $\\
\hline
$7$ & $334972400$ & $609422800 $\\
\hline
$49$ & $-1126868422025500700$ & $-391120313742441500$\\
\hline
$13$ & $91151149180$ & $-263384451140$\\
\hline
$169$ & $-299941151717771094659180$ & $323494600665947822387860$\\
\hline
\end{tabular}

\end{center}

\caption{Quelques valeurs propres des opérateurs de Hecke en genre 2 \cite{chenevier}.}

\label{proprehecke}

\end{table}

%
%
%
%
%
%

\begin{table}[h]
\begin{center}
\begin{tabular}{|c|c|c|}
\hline
 $p$ & $2$  & $3$  \\
\hline
$A_{6,8}(p)$ & $3^3.5.37.137.197.2039$   &$2^{12}.5.17.109.54196568909$ \\
\hline
$B_{6,8}(p)$ &$2^{12}.3^3.5.11^2 $ & $2^{12}.3^9.5.11^2.53$   \\
\hline
$C_{6,8}(p)$ & $-2^{14}.3^{2}.5^2.13^2.229$& $-2^{14}.3^{10}.5^2.3251.5741$  \\
\hline
$D_{6,8}(p)$ &$-2^{6}.3^{2}.5.2731.2939.4567$ & $-2^{14}.3^{6}.5.11.116167.61722049878337$ \\
\hline
\end{tabular}
\end{center}
\caption{Factorisations pour $(j,k)=(6,8)$.}
\label{(6,8)}
\end{table}

\begin{table}[h]
\begin{center}
\begin{tabular}{|c|c|c|}
\hline
 $p$ & $2$  & $3$ \\
\hline
$A_{8,8}(p)$ &$3^{2}.488358740729$ &  $2^{12}.11.13.43.67.64841599673 $  \\ 
\hline
$B_{8,8}(p)$ &$2^{13}.3^2.5^2.13.43$ & $ 2^{17}.3^8.5^2.11.13.67 $ \\
\hline
$C_{8,8}(p)$ & $-2^{13}.3^2.17.23^{3}.37$  & $-2^{13}.3^9.23^2.67.711097  $ \\
\hline 
$D_{8,8}(p)$ & $-2^6.3^2.5^3.7.301577.1891283$  &   $-2^9.3^6.5.41.1019.12011419289294459$ \\ 
\hline

\end{tabular}
\end{center}
\caption{Factorisations pour $(j,k)=(8,8)$.}
\label{(8,8)}
\end{table}

\begin{table}[h]
\begin{center}
\begin{tabular}{|c|c|c|}
\hline
 $p$ & $2$  & $3$ \\
\hline
 $A_{12,6}(p)$ & $3^3.5.32581834951$  & $2^{14}.5.11.121424757430201$ \\   
\hline
$B_{12,6}(p)$ &$2^{13}.3^5.5.7.13.19 $  & $ 2^{13}.3^8.5.7.13.681589  $ \\
\hline
$C_{12,6}(p)$& $-2^{14}.3^3.7.13.257$ & $-2^{14}.3^9.7.13.107.1801  $ \\
\hline
$D_{12,6}(p) $ & $ -2^4.3^3.5.7.612083.128494763$   &$-2^8.3^4.5.13.757.41442159469563851813  $ \\
\hline
\end{tabular}

\end{center}
\caption{Factorisations pour $(j,k)=(12,6)$.}
\label{(12,6)}
\end{table}

\begin{table}[h]
\begin{center}
\begin{tabular}{|c|c|c|}
\hline
 $p$ & $2$ & $3$ \\
\hline
$A_{4,10}(p) $ & $ 3^3.5^2.11.10861.54581$ & $2^{11}.5^2.11.41.103.739.62253281  $ \\  
\hline 
$B_{4,10}(p) $ & $2^{12}.3^2.5.11.41$   &$2^{11}.3^{12}.5.11.41  $ \\
\hline 
$C_{4,10}(p)$ & $-2^{14}.3^4.11.67.36137$ &$     -2^{13}.3^{15}.11.17^3.173.2689  $ \\
\hline
$D_{4,10}(p)$ & $-2^6.3^2.5.19.25841.51063917$ & $    -2^8.3^8.5^2.1151.10828398485532941  $\\
\hline
\end{tabular}
\end{center}
\caption{Factorisations pour $(j,k)=(4,10)$.}
\label{(4,10)}
\end{table}

\newpage

\bibliographystyle{siam}
\bibliography{bibliographie}
\end{document}